\documentclass[11pt,reqno]{amsart}

\usepackage[T1]{fontenc}
\usepackage[utf8]{inputenc}
\usepackage[a4paper,margin=1.15in]{geometry}
\usepackage[osf]{mathpazo}
\usepackage{microtype}

\usepackage{amsmath,amssymb,amsthm,mathtools}

\usepackage{aliascnt}
\usepackage{enumitem}
\usepackage{xcolor}
\definecolor{RoyalBlue}{RGB}{30,70,150}
\usepackage[colorlinks=true,
            linkcolor=RoyalBlue,
            citecolor=RoyalBlue,
            urlcolor=RoyalBlue]{hyperref}
\hypersetup{
  pdftitle={The sharp 2/3 exponent in the planar minimal distance problem},
  pdfauthor={Cosmin Pohoata},
  pdfsubject={Point--line separation via the positive trace form in totally real fields},
  pdfkeywords={minimal distance problem, point--line incidences, Szemeredi--Trotter theorem, totally real number fields, positive trace form, Furstenberg--Sarkozy theorem, norm}
}
\usepackage{xurl}
\setlist[itemize]{leftmargin=1.8em,itemsep=0.3em,topsep=0.3em}
\setlist[enumerate]{leftmargin=1.9em,itemsep=0.3em,topsep=0.3em}

\newtheorem{theorem}{Theorem}[section]

\newaliascnt{proposition}{theorem}
\newtheorem{proposition}[proposition]{Proposition}
\aliascntresetthe{proposition}

\newaliascnt{lemma}{theorem}

\aliascntresetthe{lemma}

\newaliascnt{observation}{theorem}
\newtheorem{observation}[observation]{Observation}
\aliascntresetthe{observation}

\newaliascnt{fact}{theorem}
\newtheorem{fact}[fact]{Fact}
\aliascntresetthe{fact}

\newaliascnt{corollary}{theorem}
\newtheorem{corollary}[corollary]{Corollary}
\aliascntresetthe{corollary}

\newaliascnt{conjecture}{theorem}
\newtheorem{conjecture}[conjecture]{Conjecture}
\aliascntresetthe{conjecture}

\theoremstyle{definition}
\newaliascnt{definition}{theorem}

\aliascntresetthe{definition}

\theoremstyle{remark}
\newaliascnt{remark}{theorem}
\newtheorem{remark}[remark]{Remark}
\aliascntresetthe{remark}

\usepackage[nameinlink,capitalise]{cleveref}

\newcommand{\R}{\mathbb{R}}
\newcommand{\Q}{\mathbb{Q}}
\newcommand{\Z}{\mathbb{Z}}
\newcommand{\OK}{\mathcal{O}_K}
\newcommand{\Tr}{\operatorname{Tr}}
\newcommand{\Norm}{\operatorname{N}}

\newcommand{\eps}{\varepsilon}
\newcommand{\dist}{\operatorname{dist}}
\newcommand{\PL}{\Delta_{\mathrm{PL}}}

\crefname{equation}{}{}
\Crefname{equation}{}{}
\crefname{theorem}{Theorem}{Theorems}
\Crefname{theorem}{Theorem}{Theorems}
\crefname{proposition}{Proposition}{Propositions}
\Crefname{proposition}{Proposition}{Propositions}
\crefname{lemma}{Lemma}{Lemmas}
\Crefname{lemma}{Lemma}{Lemmas}
\crefname{corollary}{Corollary}{Corollaries}
\Crefname{corollary}{Corollary}{Corollaries}
\crefname{definition}{Definition}{Definitions}
\Crefname{definition}{Definition}{Definitions}
\crefname{remark}{Remark}{Remarks}
\Crefname{remark}{Remark}{Remarks}

\title[The sharp exponent for the minimal distance problem]
      {The sharp exponent for the minimal distance problem}
\author{Cosmin Pohoata}
\address{Department of Mathematics, Emory University, Atlanta, GA, USA}
\email{cosmin.pohoata@emory.edu}
\date{}

\begin{document}

\begin{abstract}
We show that for every fixed $\eps>0$, there exist arbitrarily large families of point--line pairs $(x_1,\ell_1),\ldots,(x_n,\ell_n)$ in $[0,1]^2$, with $x_i \in \ell_i$ for all $i$, and such that
\[
        \dist(x_i,\ell_j)\ge n^{-2/3-\eps}\qquad(i\ne j).
\]
Combined with a previous result of Cohen, the author, and Zakharov, this solves the minimal distance problem.

The same construction also comes with an unexpected finite field consequence: for every $\eps>0$, there exists a set of primes $q$ of positive relative density for which $\mathbb F_q^2$ contains an induced point-line matching of size $\gtrsim_\eps q^{3/2-\eps}$. This disproves a conjecture of Hunter, the author, Verstra\"ete and Zhang. 
\end{abstract}

\maketitle

\section{Introduction}
\label{sec:intro}

Choose points $x_1,\ldots,x_n\in[0,1]^2$ and, for each $i=1,\ldots,n$, let $\ell_i$ be a line through $x_i$. The minimal distance problem asks how far every
point can be kept from every line to which it is not assigned. More
precisely, we define
\begin{equation}
\label{eq:pl-definition}
        \PL(n):=
        \sup \min_{i\ne j}\dist(x_i,\ell_j),
\end{equation}
where the supremum is over all configurations of points $x_1,\ldots,x_n\in[0,1]^2$ and lines $\ell_1,\ldots,\ell_n \subset \mathbb{R}^{2}$ such that $x_i \in \ell_i$. For convenience, let us call these {\it{point-line incidence configurations}} of size $n$.  The quantity $\PL(n)$ can therefore be regarded as the largest scale at which one can realize $n$ prescribed point--line incidences while avoiding every nontrivial incidence. In \cite{CPZ2025}, Cohen, the author, and Zakharov introduced the problem of determining the asymptotics of $\PL(n)$ as $n$ grows, in connection with the Heilbronn triangle problem. 

Trivially, taking $x_i=(i/n,0)$ and $\ell_i=x_i+(0,1)\R$, for $i=1,\ldots n$, gives a point-line incidence configuration with $\dist(x_i,\ell_j) \ge 1/n$ for all $i \neq j$. Thus, $\PL(n) \geq 1/n$. In the other direction, it is not difficult to see that in any point-line incidence configuration of size $n$, we must have
$\dist(x_i,\ell_j)\le |x_i-x_j| \lesssim n^{-1/2}$ for some $i \neq j$. Hence $\PL(n)\lesssim n^{-1/2}$. In \cite{CPZ2025}, Cohen-Pohoata-Zakharov improved this upper bound to 
\begin{equation} \label{eq:CPZ}
\PL(n)\le n^{-2/3+o(1)},
\end{equation}
and used this estimate to establish the latest record for the Heilbronn triangle problem. For more history and context, see this recent survey by Zakharov \cite{Zakharov2026}. 

The connection itself between the Heilbronn triangle problem and the minimal distance problem is simple enough that we can recall it here: if we let
\[
        \Delta_{\mathrm H}(n)
        :=\sup_{\substack{P\subset[0,1]^2\\ |P|=n}}
          \ \min_{\substack{p,q,r\in P\\ \text{distinct}}}
          \operatorname{area}(p,q,r)
\]
denote the largest possible area of the smallest triangle determined by $n$ points in the unit square, then a simple iterated pigeonhole argument gives, from any $n$-point set $P$, at least $m\gtrsim n$ disjoint pairs $\{p_i,q_i\}$ with $|p_i-q_i|\lesssim n^{-1/2}$. For each $i=1,\ldots,m$, one can then define $\ell_i$ to be the line through $p_i$ and $q_i$. For some $i\ne j$ we have
$\dist(p_i,\ell_j)\le \PL(m)$, and therefore
\[
        \operatorname{area}(p_i,p_j,q_j)
        =\frac12|p_j-q_j|\dist(p_i,\ell_j)
        \lesssim n^{-1/2}\PL(m).
\]
Since $m\gtrsim n$ and $\PL$ is nonincreasing (restricting a given point-line incidence configuration to a smaller configuration cannot decrease its minimum separation), this immediately translates to the fact that there is an absolute constant $c>0$ such that $\Delta_{\mathrm H}(n) \lesssim n^{-1/2}\PL(\lfloor cn\rfloor)$. Thus \eqref{eq:CPZ} yields $\Delta_{\mathrm H}(n)\le n^{-7/6+o(1)}$, the current best known upper bound for the Heilbronn problem.

That being said, the Heilbronn triangle problem is not the only reason to study $\PL$. For example, in a recent paper \cite{HPVZ2026}, Hunter, the author, Verstra\"ete, and Zhang also established several new connections between the minimal distance problem and various other topics in finite geometry and arithmetic combinatorics. A particularly surprising one is with the so-called Furstenberg--S\'ark\H{o}zy problem. This question concerns the quantity
\begin{equation}
\label{eq:sN-definition}
        s(N)
        :=
        \max\Bigl\{|A|:A\subset[N],\ 
        (A-A)\cap\{m^2:m\in\Z\setminus\{0\}\}=\varnothing\Bigr\}.
\end{equation}
In other words, $s(N)$ is the largest size of a subset of $[N]$ containing no two distinct elements whose difference is a perfect square.  Furstenberg
\cite{Furstenberg1977} and S\'ark\H{o}zy \cite{Sarkozy1978}
independently proved that $s(N)=o(N)$.  Furstenberg's argument was ergodic, whereas S\'ark\H{o}zy's circle-method proof was quantitative. This estimate was subsequently sharpened by Pintz--Steiger--Szemer\'edi \cite{PintzSteigerSzemeredi1988} and, more recently, by Bloom--Maynard \cite{BloomMaynard2022} and Green--Sawhney \cite{GreenSawhney2024}, with the current record being at $s(N)\lesssim N\exp\!\bigl(-c\sqrt{\log N}\bigr)$ for some absolute constant $c>0$. 

In \cite[Theorem~1.15 and Proposition~2.3]{HPVZ2026}, Hunter--Pohoata--Verstra\"ete--Zhang proved the following inequality directly connecting the Furstenberg-S\'ark\H{o}zy problem with the minimal distance problem:
\begin{equation}
\label{eq:sN-to-PL}
        \PL\!\left(\left\lfloor cN^{1/2}s(N)\right\rfloor\right)
        \gtrsim N^{-1}.
\end{equation}
Here $c>0$ denotes an absolute constant. Hence any improvement beyond the exponent $2/3$ in \eqref{eq:CPZ} would, in principle, have a substantial consequence for the Furstenberg--S\'ark\H{o}zy
problem (to be precise, an estimate of the form $\PL(n)\le n^{-2/3-\alpha+o(1)}$ for some $\alpha>0$ would immediately imply the power-saving estimate $s(N) \le N^{1-\frac{9\alpha}{4+6\alpha}+o(1)}$).

Conversely, the relation \eqref{eq:sN-to-PL} also immediately converts lower bounds for $s(N)$ into lower bounds for $\PL$ (the original motivation in \cite{HPVZ2026}). In a celebrated paper, Ruzsa
\cite{Ruzsa1984} proved that if $m$ is squarefree and
$D\subset\Z/m\Z$ has no nonzero square difference modulo $m$, then
\begin{equation}
\label{eq:Ruzsa-lower}
        s(N)\gtrsim_{m,D}
        N^{\frac12\left(1+\frac{\log|D|}{\log m}\right)}.
\end{equation}
His choice $|D|=7$ modulo $65$ gave the exponent
$0.733077\ldots$.  Refinements due to Beigel--Gasarch and Lewko give
the best presently known explicit lower bound of $s(N)\gtrsim N^{\beta}$, where $\beta:=\frac12\left(1+\frac{\log 12}{\log 205}\right) \approx 0.733411797$. See \cite{BeigelGasarch2008,Lewko2015}. Substitution into
\eqref{eq:sN-to-PL} then gave the following lower bound in \cite{HPVZ2026}:
\begin{equation}
\label{eq:HPVZ-lower}
        \PL(n)\gtrsim n^{-\gamma},
        \qquad
        \gamma
        :=\frac1{\beta+1/2}
        =0.810759230\ldots.
\end{equation}
This result improved an earlier construction of Logunov--Zakharov
\cite{LogunovZakharov2025}, and prior to this work represented the best known lower bound for $\PL$. 

Our main result closes the gap between \eqref{eq:HPVZ-lower} and \eqref{eq:CPZ}. 

\begin{theorem}
\label{thm:main}
For every $\eps>0$, there exists $n_0(\eps)$ such that, for every integer
$n\ge n_0(\eps)$, there are points $p_1,\ldots,p_n\in[0,1]^2$ and real
lines $\ell_1,\ldots,\ell_n$, with $p_i\in\ell_i$ for every $i$, such
that
\[
        \dist(p_i,\ell_j)\ge n^{-2/3-\eps}
        \qquad(i\ne j).
\]
Equivalently,
\[
        \PL(n)\ge n^{-2/3-\eps}
        \qquad(n\ge n_0(\eps)).
\]
\end{theorem}

Combining \cref{thm:main} with \eqref{eq:CPZ}, it thus follows that the correct exponent for the minimal distance problem is $2/3$. 

\begin{corollary}
\label{cor:sharp}
As $n\to\infty$,
\[
        \PL(n)=n^{-2/3+o(1)}.
\]
\end{corollary}

The same construction also comes with an unexpected finite field consequence. For a prime power $q$, let $\operatorname{IM}(2,q)$ denote the largest
$n$ for which there exist points
$p_1,\dots,p_n\in\mathbb F_q^2$ and affine lines
$\ell_1,\dots,\ell_n\subset\mathbb F_q^2$ such that $p_i\in\ell_j$ if and only if $i=j$. As discussed in \cite{CPZ2023,CPZ2025,HPVZ2026}, this is a natural finite-field model for the minimal distance problem: avoidance of off-diagonal incidences over $\mathbb F_q$ plays the role of quantitative point--line separation over $\mathbb R$. A simple Cauchy-Schwarz incidence estimate gives $\operatorname{IM}(2,q)\le q^{3/2}+q$, for every prime power $q$. This may be viewed as the finite-field
precursor of \eqref{eq:CPZ}.

When $q$ is a square, the Hermitian unital gives an induced matching of size $q^{3/2}$, showing that this estimate is sharp up, to lower-order terms. See for example \cite{HPVZ2026} for a more elaborate discussion. For fields of prime order $q$, the situation is significantly more mysterious. In \cite[Theorem~1.2]{HPVZ2026}, Hunter--Pohoata--Verstra\"ete--Zhang used \eqref{eq:Ruzsa-lower} to prove that $\operatorname{IM}(2,q)\gtrsim q^{1.2334\ldots}$ holds for all primes $q$, thereby establishing a separation between this problem and the closely related problem of determining the size of the largest clique in the Paley graph of prime order $q$. Furthermore, in \cite[Conjecture~10.2]{HPVZ2026}, the authors also subsequently conjectured that $\operatorname{IM}(2,q) \leq q^{3/2-c}$ must hold for some absolute constant $c > 0$, for all sufficiently large primes $q$. 

Using the construction behind Theorem \ref{thm:main}, we can show that this conjecture is false, as originally stated. 

\begin{theorem}
\label{thm:prime-field-matching}
For every $\eps>0$, there is a set $\mathcal Q_\eps$ of rational
primes having positive relative density among the primes such that,
for every sufficiently large $q\in\mathcal Q_\eps$,
$$\operatorname{IM}(2,q)
        \gtrsim_\eps q^{3/2-\eps}.$$
\end{theorem}

More precisely, what we will show that for every fixed prime $r\ge5$, we must always have that $\operatorname{IM}(2,q) \gtrsim_r q^{3/2-\frac{2}{r-1}}$ for every sufficiently large prime $q\equiv\pm1\pmod r$. By the prime number theorem in arithmetic progressions, these primes have relative density $2/(r-1)$ among the primes.

\medskip

\noindent {\bf{Overview}}.  The proof of Theorem \ref{thm:main} will build upon the ideas from \cite{HPVZ2026}. Roughly speaking, the driving force behind the relation \eqref{eq:sN-to-PL} is a very simple polynomial identity: if $Q(x,y)=2x-y^2 \in \mathbb{Z}[x,y]$, then
\[
        Q\bigl((x,y)+t(y,1)\bigr)=Q\bigl(x+ty,y+t\bigr)=Q(x,y)-t^2,\ \ \text{for every}\ t \in \mathbb{Z}.
\]
This has the following crucial consequence. If $A \subset [N]$ is a set of integers without perfect square differences, and we let $P = \left\{(x,y) \in \mathbb{R}^{2}: Q(x,y) \in A\right\}$, then the ``integer line"
$\ell_{x,y} = \left\{(x,y)+t(y,1), \ t \in \mathbb{Z}\right\}$ can only intersect $P$ at the point $(x,y)$ itself. After further reflection, this allows one to convert a square-difference-free set $A \subset [N]$ of $Q$-values into an induced point--line matching in the Euclidean plane. A direct computation can then subsequently show that such an induced point-line matching is in fact a point-line incidence configuration with strong Euclidean separation (in the sense of \eqref{eq:pl-definition}). Over the integers, the quality of this separation will constrained by the Furstenberg--S\'ark\H{o}zy problem and by the density of the Ruzsa sets $A \subset [N]$. 

For $s(N)$ constructions, it is well-known that the natural endpoint of the Ruzsa digit method is $\beta=3/4-o(1)$, which \eqref{eq:sN-to-PL} would readily convert into the exponent $\gamma=4/5+o(1)$ for the minimal distance problem. In order to use \eqref{eq:Ruzsa-lower} to show that $s(N) \gtrsim N^{3/4-o(1)}$, one would however need to find, for squarefree moduli $m$, modular square-difference-free sets of size $m^{1/2-o(1)}$. At a single prime this becomes a large independent-set problem in the Paley graph and square-root-size examples are not expected. Composite moduli could in principle allow for more elaborate entangled constructions, but in practice the best known exponent remains $\approx 0.733411797$ (and achieving the $3/4-o(1)$ limit for $s(N)$ might not even be possible with this method to begin with). 

The proof of Theorem \ref{thm:main} bypasses this barrier and the Ruzsa limit altogether, by performing the same construction (from scratch) in a high degree totally real number field instead. A key idea will be to replace the Ruzsa set by the trace-zero slice of its ring of integers. The positivity of the quadratic form $\Tr_{K/\Q}(z^2)$ rules out nonzero square differences in this slice, while the field norm converts algebraic nonvanishing into quantitative Euclidean separation after one real embedding. 

\begin{remark}
Curiously, the high degree number fields enter the proof of Theorem \ref{thm:main} in a very different manner than how they do in the recent constructions for the unit distance problem \cite{AlonEtAl2026,LeePohoataZhu2026}, the sum-product problem over the reals \cite{BSSZ2026}, and the Elekes-R\'onyai problem \cite{PohoataER2026}. In all these other arguments, a local gain is (vertically) amplified across many residue-field, unit, or archimedean coordinates, and one needs arithmetic control, e.g. bounded root discriminant and/or split primes, to keep the ambient lattice from overwhelming that gain. See for example \cite{PohoataER2026, Bloom, Tao} and the references therein for more systematic discussions. Here the field is fixed before the box grows, and no sophisticated tower, split-prime input, or uniform regulator estimate is needed. The trace equation removes only one global dimension, and the degree increase will be used solely to dilute the effect of this fixed codimension. Thus, for example, if $\zeta_p=e^{2\pi i/p}$ denotes a primitive $p$th root of unity, it suffices to take $K=\Q(\zeta_p+\zeta_p^{-1})$ for a sufficiently large prime $p$.
\end{remark}

\medskip

\noindent {\bf{Paper organization}}. The rest of the paper is organized as follows. In \cref{sec:parabola} we will review the construction from \cite{HPVZ2026} in full detail, with all the computations done explicitly. In
\cref{sec:trace} we will construct the trace-zero sets. In
\cref{sec:lift} we will discuss the new number-field construction and the proof of Theorem \ref{thm:main}. In \cref{sec:split-primes} we will discuss the proof of Theorem \ref{thm:prime-field-matching}. 

\section{The Hunter-Pohoata-Verstra\"ete-Zhang construction}
\label{sec:parabola}

With an eye on the number field story that will follow, let us first isolate some basic facts that hold in greater generality. 

\subsection{The induced-matching identity}

Let $F$ be any field of characteristic different from $2$, and let us $Q(x,y):=2x-y^2$ as a polynomial in $F[x,y]$. In particular, the level set $Q(x,y)=a$ is a parabola in $F^2$.  For a point $p=(x,y)\in F^2$,
let
\begin{equation}
\label{eq:algebraic-line}
        \mathcal L_p
        :=p+F(y,1)
        =\{(x+ty,y+t):t\in F\}.
\end{equation}
This is the tangent line to the parabola $Q=Q(p)$ at $p$. For two points $p=(x,y)$ and $p'=(x',y')$, let us also single out the following recurring quantity:
\begin{equation}
\label{eq:D-definition}
        D(p,p'):=x-x'+y(y'-y).
\end{equation}
We next record the key identity mentioned above, along with some important consequences.

\begin{observation}
\label{lem:parabola-identity}
For $p=(x,y)\in F^2$ and $t\in F$, $Q\bigl(p+t(y,1)\bigr)=Q(p)-t^2$. Moreover, for $p'=(x',y')$, we have that $Q(p)-Q(p')=2D(p,p')+(y'-y)^2$,
and
\begin{equation}
\label{eq:line-incidence-D}
        p'\in\mathcal L_p
        \quad\Longleftrightarrow\quad
        D(p,p')=0.
\end{equation}
\end{observation}

In particular, Observation \ref{lem:parabola-identity} gives us the following important implication:
\begin{equation}
\label{eq:parabola-incidence}
        p'\in\mathcal L_p
        \quad\Longrightarrow\quad
        Q(p)-Q(p')=(y'-y)^2.
\end{equation}

For the sake of completeness, we include the easy proof as well.

\begin{proof}
The parametrisation in \eqref{eq:algebraic-line} gives
\[
\begin{aligned}
        Q(x+ty,y+t)=2(x+ty)-(y+t)^2=2x-y^2-t^2,
\end{aligned}
\]
which gives the first equality. Expanding $2D(p,p')+(y'-y)^2$ gives
\[
        2x-2x'+2y(y'-y)+(y'-y)^2
        =2x-y^2-(2x'-y'^2)=Q(p)-Q(p').
\]
Finally, a defining equation for $\mathcal L_p$ is
\begin{equation}
\label{eq:unscaled-line-equation}
        \Lambda_p(X,Y):=X-yY+(y^2-x)=0.
\end{equation}
Its value at $p'$ is $\Lambda_p(p')=x'-yy'+y^2-x=-D(p,p')$, so it vanishes exactly when $p'\in\mathcal L_p$.
\end{proof}

A useful take-away is that the quantity $D(p,p')$ is precisely the negative of the value at
$p'$ of the normalized affine equation \eqref{eq:unscaled-line-equation} for $\mathcal L_p$ (so it detects incidences). For any $A,Y\subset F$, let us next define the set
\begin{equation}
\label{eq:P-A-Y}
        \mathcal P(A,Y)
        :=\{(x,y)\in F^2:y\in Y,\ Q(x,y)\in A\}.
\end{equation}
Because the characteristic is not $2$, the map
\begin{equation}
\label{eq:A-Y-bijection}
        A\times Y\longrightarrow\mathcal P(A,Y),
        \qquad
        (a,y)\longmapsto\left(\frac{a+y^2}{2},y\right),
\end{equation}
is a bijection. To take advantage of \eqref{eq:parabola-incidence}, recall that the set $A \subset F$ will be chosen to be a set with $(A-A)\cap\{t^2:t\in F^\times\}=\varnothing$. For convenience, let us call such sets \emph{square-difference-free}. 

\begin{corollary}
\label{cor:algebraic-matching}
If $A\subset F$ is square-difference-free, then the pairs $\{(p,\mathcal L_p):p\in\mathcal P(A,Y)\}$ form an induced matching: for $p,p'\in\mathcal P(A,Y)$,
\[
        p'\in\mathcal L_p
        \quad\Longleftrightarrow\quad
        p'=p.
\]
\end{corollary}

\begin{proof}
Suppose that $p'\in\mathcal L_p$.  By
\eqref{eq:parabola-incidence}, the difference $Q(p)-Q(p')$ is a square.
Since both values lie in $A$, square-difference-freeness gives
$Q(p)=Q(p')$.  Hence $(y'-y)^2=0$, so $y'=y$, and also $x'=x$.
\end{proof}

\subsection{Euclidean separation}

We now specialise to the integer model and compute explicitly the Euclidean separation.  Let
$q$ be a positive integer and let $A\subset[q^2]$ be
square-difference-free in $\Z$.  Set $\mathcal P_q(A):=\left\{\left(\frac{a+y^2}{2},y\right): a\in A,\ y\in[q]\right\}$. For every point $p=(x,y)\in\mathcal P_q(A)$, regard $\mathcal L_p$ as a real line $\mathcal L_p=\{(x+ty,y+t):t\in \mathbb{R}\}$, and apply the anisotropic linear map $\phi_q(x,y):=\left(\frac{y}{q},\frac{x}{q^2}\right)$. Finally, write $p^*:=\phi_q(p)$ and $\ell_p^*:=\phi_q(\mathcal L_p)$. We can now check the following:

\begin{proposition}
\label{prop:HPVZ-matching}
The set $\mathcal P_q(A)$ has $q|A|$ elements, every $p^*$ lies in
$[0,1]^2$, and $p^*\in\ell_p^*$.  For
$p=(x,y),p'=(x',y')\in\mathcal P_q(A)$,
\begin{equation}
\label{eq:integer-distance}
        \dist(p'^*,\ell_p^*)
        =
        \frac{|D(p,p')|}{\sqrt{q^4+q^2y^2}}.
\end{equation}
If $p\ne p'$, then
\begin{equation}
\label{eq:HPVZ-separation}
        \dist(p'^*,\ell_p^*)
        \ge\frac{1}{2\sqrt2\,q^2}.
\end{equation}
Thus the pairs $(p^*,\ell_p^*)$ form a robust induced matching.
\end{proposition}

In particular, \cref{prop:HPVZ-matching} gives $\PL\bigl(q \cdot s(q^2)\bigr)\ge\frac{1}{2\sqrt2\,q^2}$, which recovers \eqref{eq:sN-to-PL}.

\begin{proof}
By design, the parametrisation in the definition of $\mathcal P_q(A)$ is injective, so the point set has $q|A|$ elements.  Also
$1\le y\le q$ and $0<x=\frac{a+y^2}{2}\le q^2$, which proves $p^*\in[0,1]^2$.  The assigned incidence is preserved by
the invertible map $\phi_q$.

Write $(S,T)=\phi_q(X,Y)$.  Under this change of variables, the line
equation \eqref{eq:unscaled-line-equation} becomes
\[
        q^2T-yqS+(y^2-x)=0,
\]
whose normal vector is $(-yq,q^2)$.  Its value at $p'^*$ is
$\Lambda_p(p')=-D(p,p')$.  The point-to-line distance formula therefore
gives \eqref{eq:integer-distance}.

Let $a=Q(p)$ and $a'=Q(p')$.  By Observation \ref{lem:parabola-identity},
\begin{equation}
\label{eq:2D-integer}
        2D(p,p')=a-a'-(y'-y)^2.
\end{equation}
If $D(p,p')=0$, square-difference-freeness forces $p=p'$.  Thus, for $p\ne p'$, the right-hand side
of \eqref{eq:2D-integer} is a nonzero integer, and
$|D(p,p')|\ge1/2$.  Since $y\le q$, the denominator in
\eqref{eq:integer-distance} is at most $\sqrt2q^2$, proving
\eqref{eq:HPVZ-separation}.
\end{proof}

\subsection{A Szemer\'edi--Trotter configuration in the background}

For a brief moment, let us not pass to the square-difference-free set $A \subset [q^2]$ in the construction from the previous section, and consider instead the full set of points $\mathcal P_q:=\mathcal P_q([q^2])$, together with the set of lines $\mathcal L_q:=\{\mathcal L_p:p\in\mathcal P_q\}$.

Note that both sets $\mathcal P_q$ and $\mathcal L_q$ have $q^3$ elements (Indeed, $p=(x,y)$ determines
$a=Q(p)$, while the normalised equation
\eqref{eq:unscaled-line-equation} determines first $y$ and then $x$). Moreover, it turns out that $\mathcal P_{q}$ and $\mathcal L_{q}$ form an extremal Szemer\'edi-Trotter configuration.

\begin{proposition}
\label{prop:ST-grid}
The number of incidences between $\mathcal P_q$ and $\mathcal L_q$ is
\begin{equation}
\label{eq:ST-count}
\begin{aligned}
        I(\mathcal P_q,\mathcal L_q)=\sum_{|h|<q}(q-|h|)(q^2-h^2)=\frac{q^2(5q^2+1)}6
        \asymp q^4.
\end{aligned}
\end{equation}
\end{proposition}

\begin{proof}
Write $a=Q(p)$ and $a'=Q(p')$.  By
\eqref{eq:parabola-incidence}, an incidence $p'\in\mathcal L_p$ with
$h=y'-y$ requires $a-a'=h^2$.  There are $q-|h|$ ordered pairs
$(y,y')\in[q]^2$ with difference $h$, and $q^2-h^2$ ordered pairs
$(a,a')\in[q^2]^2$ with $a-a'=h^2$.  Summing over $|h|<q$ gives the
first expression in \eqref{eq:ST-count}.  Moreover,
\[
\begin{aligned}
        \sum_{|h|<q}(q-|h|)(q^2-h^2)
        &=q^3+2\sum_{h=1}^{q-1}(q-h)(q^2-h^2)\\
        &=q^3+\frac{q^2(q-1)(5q-1)}6\\
        &=\frac{q^2(5q^2+1)}6.
\end{aligned}
\]
\end{proof}

Writing $N=q^3$, we thus have that $I(\mathcal P_q,\mathcal L_q) \asymp N^{4/3}$, thereby $\mathcal P_q$ and $\mathcal L_q$ achieve equality in the Szemer\'edi--Trotter theorem \cite{SzemerediTrotter1983}. Notably, the same configuration also achieves the critical
positive distance scale discussed in \cite[Section 3]{CPZ2023}. Indeed, if $p'\notin\mathcal L_p$, then
$2D(p,p')$ is a nonzero integer, so
\eqref{eq:integer-distance} gives a lower bound of $\asymp q^{-2}$. Therefore
\[
        \min_{\substack{p\in\mathcal P_q,\ L\in\mathcal L_q\\p\notin L}}
        \dist\bigl(\phi_q(p),\phi_q(L)\bigr)
        \asymp q^{-2}=N^{-2/3}.
\]
Thus the unrestricted parabola model is simultaneously a sharp
incidence configuration and a configuration at the critical Euclidean scale. The square-difference-free set restriction is used to only extract an induced matching from it. The next section will essentially replace this sparse integer restriction with a more efficient (codimension one) trace-zero lattice restriction over a totally real field of large degree.

\section{Trace-zero sets}
\label{sec:trace}

We use only standard facts about trace, norm, and the Minkowski embedding
of a number field. Let $K/\Q$ be
a totally real number field of degree $d\ge2$, let $\OK$ be its ring of
integers, and list its real embeddings as $\sigma_1,\ldots,\sigma_d:K\hookrightarrow\R$. Write $\iota(a):=(\sigma_1(a),\ldots,\sigma_d(a)) \in \mathbb{R}^{d}$ for the Minkowski embedding. For $u,v\in K$, define bilinear form $\langle u,v\rangle_{\mathrm{tr}}:=\Tr_{K/\Q}(uv)$. In other words,
\begin{equation}
\label{eq:trace-inner-product}
        \langle u,v\rangle_{\mathrm{tr}}
        :=\sum_{j=1}^d\sigma_j(u)\sigma_j(v)
        =\langle\iota(u),\iota(v)\rangle_{\R^d}.
\end{equation}
We record the following crucial consequence:

\begin{observation}
\label{lem:positive-trace-form}
The map $\langle\cdot,\cdot\rangle_{\mathrm{tr}}$ extends to a positive-definite inner product on
$K\otimes_{\Q}\R$, and its associated quadratic form satisfies
\begin{equation}
\label{eq:positive-trace-form}
        q_K(z):=\Tr_{K/\Q}(z^2)
        =\|\iota(z)\|_2^2>0
        \qquad(z\ne0).
\end{equation}
\end{observation}

This is where the total reality of $K$ comes in. Since
$\Tr_{K/\Q}(a)=\langle a,1\rangle_{\mathrm{tr}}$, the trace-zero subspace is the orthogonal
complement of $1$.  Equivalently, under the Minkowski embedding it is
the hyperplane $X_1+\cdots+X_d=0$. On the other hand, $\iota(z^2)=(\sigma_1(z)^2,\ldots,\sigma_d(z)^2)\in\R_{\ge0}^d$, and this vector lies in the trace-zero hyperplane only when $z=0$. Thus
\begin{equation}
\label{eq:trace-zero-no-squares}
        \ker(\Tr_{K/\Q})\cap\{z^2:z\in K\}=\{0\}.
\end{equation}
Geometrically, every nonzero square lies in the open half-space
$\Tr_{K/\Q}>0$, whose boundary is the trace-zero hyperplane.

For $X\ge0$, let us now define the symmetric Minkowski box of radius $X$
\begin{equation}
\label{eq:BK}
        B_K(X)
        :=
        \{a\in\OK:|\sigma_j(a)|\le X\text{ for }1\le j\le d\}.
\end{equation}

One of the main new ideas in this construction is to consider the trace-zero lattice
\begin{equation}
\label{eq:trace-lattice}
        \Lambda_K^0
        :=
        \{a\in2\OK:\Tr_{K/\Q}(a)=0\}.
\end{equation}
The additive group $2\OK$ has rank $d$, and the trace map is nonzero because $\Tr_{K/\Q}(2)=2d$.  Hence $\Lambda_K^0$ has rank $d-1$ (it lives in the trace-zero hyperplane
$x_1+\cdots+x_d=0$).  

For two parameters $R,M\ge1$ (to be balanced appropriately later), let us denote
\begin{equation}
\label{eq:AK-YK}
        A_K(R):=\Lambda_K^0\cap B_K(R),
        \qquad
        Y_K(M):=2\OK\cap B_K(M).
\end{equation}
We record two important facts about $A_{K}(R)$ and $Y_{K}(M)$. 
\begin{proposition}
\label{prop:trace-sets}
The set $A_K(R)$ is square-difference-free. Moreover, for every $R,M\ge1$,
\begin{equation}
\label{eq:trace-counts}
        |A_K(R)|\asymp_K R^{d-1},
        \qquad
        |Y_K(M)|\asymp_K M^d.
\end{equation}
\end{proposition}

\begin{proof}
Let's first check the square-difference-freeness. Suppose that $a-a'=z^2$ with $a,a'\in A_K(R)$ and $z\in\OK$.  Then
\[
        0
        =\Tr_{K/\Q}(a-a')
        =\Tr_{K/\Q}(z^2)
        =q_K(z).
\]
By the positive definiteness from Observation \ref{lem:positive-trace-form}, one has $z=0$, and therefore $a=a'$. 

The size estimates for $A_{K}(R)$ and $Y_{K}(M)$ immediately follow from the following standard fact about lattices:

\begin{fact}
\label{lem:lattice-count}
Let $\Lambda\subset\R^m$ be a lattice of rank $r\ge1$.  Then
\[
        |\Lambda\cap[-T,T]^m|\asymp_{\Lambda}T^r
\]
for every $T\ge1$.
\end{fact}

Therefore, $|A_K(R)|\asymp_K R^{d-1}$ comes from the fact that the lattice $\Lambda_K^0$ has rank $d-1$, whereas $|Y_K(M)|\asymp_K M^d$ follows from the fact that $2\OK$ has full rank $d$.

\end{proof}

The fixed factor $2$ in \eqref{eq:trace-lattice} only clears the
possible half-integral first coordinate in \eqref{eq:A-Y-bijection}.
If $a,y\in2\OK$, then $x=\frac{a+y^2}{2}\in\OK$. Consequently, for two points $p=(x,y)$ and $p'=(x',y')$ arising from
$A\times Y$, the quantity
\[
        D(p,p')
        =x-x'+y(y'-y)
        =\frac{a-a'-(y'-y)^2}{2}
\]
is also an algebraic integer.  Passing from $\OK$ to the finite-index
sublattice $2\OK$ has no effect on any exponent.

\begin{remark}
\label{rem:trace-paraboloid}
The square-difference-free property has a direct geometric formulation.
In $K^2$, consider the trace level set
\[
        \mathcal S_K
        :=\{(x,y)\in K^2:\Tr_{K/\Q}(Q(x,y))=0\}.
\]
Through every $p=(x,y)\in\mathcal S_K$, take the affine $K$-line $\left\{ p+t(y,1):\ t\in K\right\}$. By Observation \ref{lem:parabola-identity},
\[
        \Tr_{K/\Q}\bigl(Q(p+t(y,1))\bigr)
        =-\Tr_{K/\Q}(t^2).
\]
Furthermore, by Observation \ref{lem:positive-trace-form}, this vanishes only when $t=0$. Thus every assigned line meets the trace level set in a unique
$K$-rational point.  The next section converts this algebraic uniqueness into a quantitative separation statement, after appropriately projecting onto the real plane.
\end{remark}

\section{Point-line incidence configurations from number fields}
\label{sec:lift}

Fix a totally real field $K$ of degree $d$ and distinguish one embedding $\sigma:K\hookrightarrow\R$. The main purpose of this section is to establish the following number field analogue of Proposition \ref{prop:HPVZ-matching}.

\begin{proposition}
\label{prop:nf-lift}
Let $A\subset2\OK\cap B_K(R)$ be square-difference-free, and let
$Y\subset2\OK\cap B_K(M)$, where $R,M\ge1$. Define
\[
        \mathcal P(A,Y)
        :=
        \left\{
        \left(\frac{a+y^2}{2},y\right):
        a\in A,\ y\in Y
        \right\}
        \subset\OK^2,
\]
and put
\begin{equation}
\label{eq:N-H}
        N:=\frac{R+M^2}{2},
        \qquad
        H:=R+2M^2.
\end{equation}
Then one can associate to each $p\in\mathcal P(A,Y)$ a point
$p^*\in[0,1]^2$ and a real line $\ell_p^*$ through $p^*$ such that,
for distinct $p,p'\in\mathcal P(A,Y)$,
\[
        \dist(p'^*,\ell_p^*)
        \ge
        \frac{H^{-(d-1)}}{2\sqrt{N^2+M^4}}.
\]
\end{proposition}

\begin{proof}
The parametrization defining $\mathcal P(A,Y)$ is injective, and hence
$|\mathcal P(A,Y)|=|A||Y|$. Let
\[
        p=(x,y)
        =\left(\frac{a+y^2}{2},y\right)
        \in\mathcal P(A,Y).
\]
Since $a,y\in2\OK$, we have $x\in\OK$. For every real embedding $\rho:K\hookrightarrow\R$, we therefore have
\[
        |\rho(x)|
        \le\frac{|\rho(a)|+|\rho(y)|^2}{2}
        \le\frac{R+M^2}{2}
        =N.
\]
So $(\sigma(x),\sigma(y))$ lies in
$[-N,N]\times[-M,M]$.

Let $L_p^\sigma$ be the real line through
$(\sigma(x),\sigma(y))$ with direction $(\sigma(y),1)$, and define
\begin{equation}
\label{eq:Phi-MN}
        \Phi_{M,N}(X,Y)
        :=\left(\frac{Y+M}{2M},\frac{X+N}{2N}\right).
\end{equation}
This is the symmetric-box version of the previous anisotropic scaling
$(x,y)\mapsto(y/M,x/N)$.  Set
\begin{equation}
\label{eq:scaled-pairs}
        p^*:=\Phi_{M,N}(\sigma(x),\sigma(y)),
        \qquad
        \ell_p^*:=\Phi_{M,N}(L_p^\sigma).
\end{equation}
Then $p^*\in[0,1]^2$ and $p^*\in\ell_p^*$.

Take $p=(x,y)$ and $p'=(x',y')$ in $\mathcal P(A,Y)$, and write
$a=Q(p)$ and $a'=Q(p')$.  The line-equation value from
\eqref{eq:D-definition} is
\begin{equation}
\label{eq:algebraic-D}
\begin{aligned}
        D(p,p') =x-x'+y(y'-y)=\frac{a-a'-(y'-y)^2}{2}\in\OK.
\end{aligned}
\end{equation}
The inclusion in $\OK$ follows from $x,x',y,y'\in\OK$.  If
$D(p,p')=0$, then $a-a'=(y'-y)^2$. Square-difference-freeness gives $a=a'$, hence $y'=y$, and then
$x'=x$.  Therefore
\begin{equation}
\label{eq:D-nonzero}
        p\ne p'
        \quad\Longrightarrow\quad
        D(p,p')\ne0.
\end{equation}

For every real embedding $\rho$, the second expression in
\eqref{eq:algebraic-D} gives
\[
        |\rho(D(p,p'))|
        \le
        \frac{|\rho(a-a')|+|\rho(y'-y)|^2}{2}
        \le R+2M^2
        =H.
\]
When $p\ne p'$, the algebraic integer $D(p,p')$ is nonzero.  Its norm is
therefore a nonzero rational integer, and
\[
        1
        \le|\Norm_{K/\Q}(D(p,p'))|
        =|\sigma(D(p,p'))|
         \prod_{\rho\ne\sigma}|\rho(D(p,p'))|
        \le|\sigma(D(p,p'))|H^{d-1}.
\]
Consequently,
\begin{equation}
\label{eq:D-lower}
        |\sigma(D(p,p'))|\ge H^{-(d-1)}.
\end{equation}

Under $\Phi_{M,N}$, it is easy to check that the direction vector for the line $\ell_p^*$ is given by $v_p=\left(\frac1{2M},\frac{\sigma(y)}{2N}\right)$,
whereas $p'^*-p^* =\left( \frac{\sigma(y'-y)}{2M}, \frac{\sigma(x'-x)}{2N}\right)$. The determinant formula for point-to-line distance then yields the exact identity
\begin{equation}
\label{eq:nf-distance-exact}
\begin{aligned}
        \dist(p'^*,\ell_p^*)=\frac{|\det(p'^*-p^*,v_p)|}{\|v_p\|_2}=\frac{|\sigma(D(p,p'))|}
                {2\sqrt{N^2+M^2\sigma(y)^2}}.
\end{aligned}
\end{equation}
Combining \eqref{eq:D-lower} with $|\sigma(y)|\le M$ proves Proposition \ref{prop:nf-lift}.
\end{proof}

The natural choice is to balance the two algebraic scales by taking
$R=M^2$.

\begin{corollary}
\label{cor:balanced-construction}
Let $M\ge1$, and take
\[
        A=A_K(M^2),
        \qquad
        Y=Y_K(M).
\]
Then the construction of \cref{prop:nf-lift} contains  $|A||Y|\asymp_K M^{3d-2}$ incident point--line pairs in $[0,1]^2$, and every off-diagonal distance
is at least $(3M^2)^{-d}$. 
\end{corollary}

\begin{proof}
The count follows from \cref{prop:trace-sets}.  With $R=M^2$, the
quantities in \eqref{eq:N-H} are $N=M^2$ and $H=3M^2$.  Hence
\cref{prop:nf-lift} gives
\[
\begin{aligned}
        \dist(p'^*,\ell_p^*) \ge
        \frac{(3M^2)^{-(d-1)}}{2\sqrt{2M^4}}=\frac{3}{2\sqrt2}(3M^2)^{-d}\ge(3M^2)^{-d}.
\end{aligned}
\]
\end{proof}

We are now ready to derive Theorem \ref{thm:main}.

\begin{proof}[Proof of \cref{thm:main}]
\label{sec:sharp}

Fix $\eps>0$, and choose an odd prime $p$ so large that $\frac{8}{9p-21}<\frac{\eps}{2}$. Let $\zeta_p=e^{2\pi i/p}$ be a primitive $p$th root of unity, and consider the maximal real subfield $K:=\Q(\zeta_p+\zeta_p^{-1}) \subset\Q(\zeta_p)$. The conjugates of $\zeta_p+\zeta_p^{-1}$ are the real numbers
\[
        \zeta_p^j+\zeta_p^{-j}
        =2\cos\left(\frac{2\pi j}{p}\right),
        \qquad 1\le j\le p-1,
\]
so $K$ is totally real. Moreover, since complex conjugation has order two in $\operatorname{Gal}(\Q(\zeta_p)/\Q)$, the field $K$ has degree $d:=[K:\Q]=\frac{p-1}{2}$. See, for example, \cite{Washington1997}. Applying \cref{cor:balanced-construction} to this fixed field $K$
produces a point--line incidence configuration of size
$n\asymp_K M^{3d-2}$ whose off-diagonal distances are at least
$(3M^2)^{-d}$. Consequently,
\[
        \PL(n)\ge (3M^2)^{-d} \asymp_{K} n^{-\theta_d},\ \ \text{where}\ \ \theta_d = \frac23 + \frac{4}{9d-6}.
\]
Since $2/3+\eps-\theta_d>\eps/2>0$, the fixed constant $c_K$ is absorbed by the remaining power of $n$, which yields $\PL(n)\ge n^{-2/3-\eps}$. This completes the proof of Theorem \ref{thm:main}.
\end{proof}

\section{Induced matchings over prime fields}
\label{sec:split-primes}

The integral nature of the construction above allows us to pass back to finite fields. For a prime power $q$, recall that
$\operatorname{IM}(2,q)$ denotes the maximum size of an induced matching in the point--line incidence graph of $\mathbb F_q^2$. For any fixed prime $r\ge5$, we will show that
\begin{equation}
\label{eq:cyclotomic-IM-intro}
        \operatorname{IM}(2,q)
        \gtrsim_r q^{3/2-\frac{2}{r-1}}
\end{equation}
for every sufficiently large prime $q\equiv\pm1\pmod r$. By the prime number theorem in arithmetic progressions, these primes have relative density $2/(r-1)$, so this will directly imply Theorem \ref{thm:prime-field-matching}. 

We first isolate the main reduction mechanism for an arbitrary fixed totally real field. Let $K/\Q$ be a fixed totally real field of degree $d\ge2$. Let $M\ge1$, and set $A:=A_K(M^2)$ and $Y:=Y_K(M)$, like in Corollary \ref{cor:balanced-construction}.

\begin{proposition}
\label{prop:finite-field-reduction}
Suppose that $q$ is a rational prime for which there is a prime ideal
$\mathfrak q\subset\OK$ with $\OK/\mathfrak q\cong\mathbb F_q$ and $q>(3M^2)^d$. Then,
\[
        \operatorname{IM}(2,q)
        \ge |A||Y| \gtrsim_K M^{3d-2}.
\]
\end{proposition}

\begin{proof}
Consider the unscaled point set
\[
        \mathcal P_M
        :=
        \mathcal P(A,Y)
        =
        \left\{
        p_{a,y}
        :=
        \left(\frac{a+y^2}{2},y\right):
        a\in A,\ y\in Y
        \right\}
        \subset\OK^2.
\]
To $p=(x,y)\in\mathcal P_M$, assign the affine $K$-line $\mathcal L_p:= \{(x+ty,y+t):t\in K\}$. Fix an isomorphism $\OK/\mathfrak q\cong\mathbb F_q$, and write
$\overline z$ for the residue class of $z\in\OK$. Reduce the point and
line by setting
\[
        \overline p:=(\overline x,\overline y) \in \mathbb{F}_{q}^{2},
        \qquad
        \overline{\mathcal L}_p
        :=
        \{
        (\overline x+t\overline y,\overline y+t):
        t\in\mathbb F_q
        \} \subset \mathbb{F}_{q}^{2}.
\]
The direction vector $(\overline y,1)$ is nonzero, so
$\overline{\mathcal L}_p$ is an affine line, and
$\overline p\in\overline{\mathcal L}_p$. We claim that $\left\{\overline{p},\overline{\mathcal L}_p\right\}_{p \in \mathcal{P}_{M}}$ form an induced matching in the point-line incidence graph of $\mathbb{F}_{q}^{2}$. 

Take
distinct $p=(x,y)$ and $p'=(x',y')$ in $\mathcal P_M$, and write
$a=Q(p)$ and $a'=Q(p')$. Recall from \eqref{eq:algebraic-D} that
\[
\begin{aligned}
        D(p,p'):=x-x'+y(y'-y)=\frac{a-a'-(y'-y)^2}{2}\in\OK.
\end{aligned}
\]
Like before, this algebraic integer is nonzero and we know that for every real embedding $\rho:K\hookrightarrow\R$,
\[
\begin{aligned}
        |\rho(D(p,p'))| \le \frac{|\rho(a-a')|+|\rho(y'-y)|^2}{2}\le
        \frac{2M^2+4M^2}{2}
        =3M^2.
\end{aligned}
\]
Consequently,
\begin{equation}
\label{eq:finite-field-norm-bound}
        0<
        \left|\Norm_{K/\Q}\bigl(D(p,p')\bigr)\right|
        \le(3M^2)^d
        <q.
\end{equation}

The line equation \eqref{eq:unscaled-line-equation} remains valid
over the residue field and takes the value
$-\overline{D(p,p')}$ at $\overline p'$. Hence
\begin{equation}
\label{eq:reduced-D-incidence}
        \overline p'\in\overline{\mathcal L}_p
        \quad\Longleftrightarrow\quad
        D(p,p')\in\mathfrak q.
\end{equation}
If the latter inclusion held, then $\mathfrak q$ would divide the
principal ideal $(D(p,p'))$. Taking ideal norms would give
\[
        q=\Norm(\mathfrak q)
        \ \bigm|\
        \left|\Norm_{K/\Q}\bigl(D(p,p')\bigr)\right|,
\]
contradicting \eqref{eq:finite-field-norm-bound}. See, for example,
\cite[Chapter~I]{Neukirch1999} for these standard norm facts. Therefore every off-diagonal incidence remains absent after reduction.

This also shows that the reduced points are pairwise distinct. Indeed,
if $\overline p=\overline p'$, then
$\overline p'\in\overline{\mathcal L}_p$. Likewise, the reduced lines
are pairwise distinct: equality
$\overline{\mathcal L}_p=\overline{\mathcal L}_{p'}$ would imply
$\overline p'\in\overline{\mathcal L}_p$. We have therefore obtained an induced matching of size $|\mathcal P_M| =|A||Y| \asymp_K (M^2)^{d-1}M^d =M^{3d-2}$ (we used again \cref{prop:trace-sets}). 
\end{proof}

To show \eqref{eq:cyclotomic-IM-intro}, let $r\ge5$ be prime and consider again $K=\Q(\zeta_r+\zeta_r^{-1})$ to be the real subfield of the $r$th cyclotomic field $Q(\zeta_r)$. Recall again that this is a Galois extension of $\Q$ of degree $d:=[K:\Q]=\frac{r-1}{2}$. Recall also that $\operatorname{Gal}(\Q(\zeta_r)/\Q) \cong(\Z/r\Z)^\times$, where the class of $a\in(\Z/r\Z)^\times$ corresponds to the automorphism $\sigma_a(\zeta_r)=\zeta_r^a$. Complex conjugation is $\sigma_{-1}$, and $K_r$ is precisely its fixed field. Consequently, $\operatorname{Gal}(K/\Q) \cong(\Z/r\Z)^\times/\{\pm1\}$.

Now let $q\ne r$ be a rational prime. Since $q$ is unramified in $\Q(\zeta_r)$, its Frobenius automorphism is characterized by $\operatorname{Frob}_q(\zeta_r)=\zeta_r^q$. Under the preceding identification, it therefore corresponds to the residue class of $q$ modulo $r$. Its restriction to $K_r$ is trivial if and only if this class becomes trivial after quotienting by $\{\pm1\}$, which is equivalent to $q\equiv\pm1\pmod r$. Concretely, if $q\equiv1\pmod r$, then Frobenius is already the identity on $\Q(\zeta_r)$; if $q\equiv-1\pmod r$, then Frobenius is complex conjugation, which acts trivially on the maximal real subfield $K_r$. Finally, since $K_r/\Q$ is Galois, an unramified prime splits completely if and only if its Frobenius automorphism is trivial. Hence \[ q\text{ splits completely in }K_r \quad\Longleftrightarrow\quad q\equiv\pm1\pmod r. \] See, for example, \cite[Chapter~2]{Washington1997} for all these standard facts. 

\begin{proof}[Proof of \cref{thm:prime-field-matching}] 

Set $M:=\left\lfloor\frac{q^{1/(2d)}}2\right\rfloor$, then $(3M^2)^d \le\left(\frac34q^{1/d}\right)^d =\left(\frac34\right)^dq <q$ for all sufficiently large $q$. Hence, we can apply Proposition \ref{prop:finite-field-reduction} to say that
\[
        \operatorname{IM}(2,q) \gtrsim_r  q^{3/2-1/d}  =
        q^{3/2-\frac{2}{r-1}}
\]
for every sufficiently large prime $q\equiv\pm1\pmod r$. This proves
\eqref{eq:cyclotomic-IM-intro}, and thus also Theorem \ref{thm:prime-field-matching}.
\end{proof}

Combining \cref{thm:prime-field-matching} with
the $\operatorname{IM}(2,q) \leq q^{3/2}+q$ estimate, it follows that
\[
        \limsup_{\substack{q\to\infty\\q\ {\rm prime}}}
        \frac{\log\operatorname{IM}(2,q)}{\log q}
        =
        \frac32.
\]
In particular, no absolute power saving of the form
$\operatorname{IM}(2,q)\le q^{3/2-c}$ can hold for all sufficiently large primes $q$.

\section{Concluding remarks}
\label{sec:perspective}

Many other interesting problems remain open. For example, we do not believe that the sharpness of the $2/3$ exponent for the minimal distance problem means also that the Cohen--Pohoata--Zakharov upper bound of $\Delta_{\mathrm H}(n)\le n^{-7/6+o(1)}$ should be definitive for the Heilbronn triangle problem. In particular, note that the proof of
\[
        \Delta_{\mathrm H}(n)
        \lesssim n^{-1/2}\PL(\lfloor cn\rfloor),
\]
discussed in Section \ref{sec:intro}, discards a substantial amount of geometry. In this spirit, we would like to take the opportunity to formally state the following conjecture: 

\begin{conjecture}
There exists an absolute constant $c>0$ such that
$$\Delta_{\mathrm H}(n)\lesssim n^{-7/6-c}.$$
\end{conjecture}

Analogously, we also believe that $s(N) \leq N^{1-c}$ should hold for the Furstenberg--S\'ark\H{o}zy problem. This belief was already recorded in several other places, see e.g. \cite{GreenSawhney2024}. In light of the proof of Theorem \ref{thm:main}, it seems also natural to study the following number field generalization. For a {\it{fixed}} totally real field $K$ of degree $d \geq 1$, let
\[
        s_K(X)
        :=\max\Bigl\{|A|:A\subset B_K(X),\ 
        (A-A)\cap\{z^2:0\ne z\in\OK\}=\varnothing\Bigr\}.
\]
The case $K=\mathbb Q$ is exactly the classical Furstenberg--S\'ark\H{o}zy problem. Indeed, in this case we have that $\OK=\mathbb Z$ and $B_{\mathbb Q}(X) = \{-\lfloor X\rfloor,\ldots,\lfloor X\rfloor\}$. Therefore, after appropriately translating, it follows that $s_{\mathbb Q}(X)=s(2\lfloor X\rfloor+1)$. See also the closely related version from \cite{PandeySaha2025} for general number fields. 

For a totally real field $K$ of degree $d \geq 2$, the trace-zero lattice used in our proof of Theorem~\ref{thm:main} gives $s_K(X)\gtrsim_K X^{d-1}$. Conversely, $\mathbb Z\cdot1$ is a primitive direct summand of the
additive group $\OK$, so one may partition $B_K(X)$ into
$O_K(X^{d-1})$ fibers parallel to this copy of $\mathbb Z$.  The intersection of $A$ with each such fiber, after translation, is an ordinary square-difference-free set contained in an interval of
length $O_K(X)$. The result of \cite{GreenSawhney2024} therefore
implies
$$X^{d-1}\lesssim_K s_K(X) \lesssim_K X^d\exp\!\bigl(-c\sqrt{\log X}\bigr)$$
for an absolute constant $c>0$. It would be interesting to narrow this gap for sufficiently large $d$. 

Finally, in the spirit of \cite{HPVZ2026}, it is also natural to ask about the higher-dimensional analogues of \cref{thm:main,thm:prime-field-matching}, together with their potential applications to Nikodym sets and minimal blocking sets. We intend to address these in a separate work.

\medskip

\section*{Acknowledgments}

The author would like to acknowledge the important role of OpenAI's GPT-5.6 Pro in preparing the manuscript and in developing the proof of Theorem~\ref{thm:main}.  The author's original idea was to use a high-degree number field analogue of the Hunter--Pohoata--Verstra\"ete--Zhang construction to get from $\PL(n)\gtrsim n^{-0.81075\ldots}$ to the Ruzsa endpoint $\PL(n)\ge n^{-4/5}$. Naturally, the point was to leverage a better behaved version of the Ruzsa construction \eqref{eq:Ruzsa-lower} inside the rings of integers of such number fields. The decisive new idea of using the codimension one, square-difference-free, trace-zero lattice $\Lambda_K^0\subset 2\OK$ instead of a Ruzsa-like set, in order to subsequently upgrade the exponent $4/5$ to the sharp exponent $2/3$, is entirely due to GPT-5.6 Pro.

\end{document}